\documentclass[12pt]{article}

\usepackage{amssymb}
\usepackage{amsmath}
\usepackage{amsthm}
\usepackage{amsfonts}

\newcommand{\mmp}{\mathbb{P}}

\newcommand{\od}{\overset{d}{=}}
\newcommand{\dod}{\overset{d}{\to}}
\newcommand{\tp}{\overset{P}{\to}}

\newcommand{\me}{\mathbb{E}}
\newcommand{\mr}{\mathbb{R}}
\newcommand{\mn}{\mathbb{N}}

\newcommand{\Sym}{\frak S}

\newcommand{\lin}{\underset{n\to\infty}{\lim}}

\newcommand{\lix}{\underset{x\to\infty}{\lim}}
\newcommand{\lit}{\underset{t\to\infty}{\lim}}

\newenvironment{myproof}{\noindent {\it Proof} }{$\Box$ }
\newtheorem{thm}{Theorem}[section]
\newtheorem{lemma}[thm]{Lemma}

\theoremstyle{definition}

\theoremstyle{remark}
\newtheorem{rem}[thm]{Remark}
\begin{document}
\title{A generalization of the Erd\H{o}s-Tur\'an law for the order of random permutation}
\author{
Alexander Gnedin\footnote{Queen Mary University of London, e-mail: a.gnedin@qmul.ac.uk},
Alexander Iksanov\footnote{National Taras Shevchenko University of Kiev, e-mail: iksan@univ.kiev.ua} \;and
Alexander Marynych\footnote{National Taras Shevchenko University of Kiev, e-mail: marynych@unicyb.kiev.ua}
}

\date{\today}
\maketitle

\begin{abstract}
\noindent We consider random permutations derived by sampling from
stick-breaking partitions of the unit interval. The cycle
structure of such a permutation can be associated with the path of
a decreasing Markov chain on $n$ integers. Under certain
assumptions on the stick-breaking factor we prove a central limit
theorem for  the logarithm of the order of  the  permutation, thus
extending the classical Erd\H{o}s-Tur\'an law for the uniform
permutations and its generalization for Ewens' permutations
associated with sampling from the PD/GEM$(\theta)$ distribution
\cite{ABT}. Our approach is based on using perturbed random walks
to obtain the limit laws for the sum of logarithms of the cycle
lengths.
\end{abstract}

\section{Introduction}
Let $\Sym_n$ be the symmetric group  on $[n]:=\{1,\ldots, n\}$.
The order of permutation $\sigma\in \Sym_n$ is the smallest positive integer $k$ such that
the $k$-fold composition of $\sigma$ with itself
is the identity permutation. The order can be determined from the
cycle representation of $\sigma$ as the least common multiple
({\rm l.c.m.}) of the cycle lengths. For instance, permutation
$\sigma=(1\,9\,6\,2)(3\,7\,5)(4\, 8)$ has order 12.

A  random permutation $\Pi_n$ of $[n]$ is a random variable
with values in the set $\Sym_n$.
A widely known parametric family of random permutations has  probability mass function
\begin{equation}\label{ESF1}
{\mathbb P}\{\Pi_n=\sigma\}=c^{-1}\,{\theta^{|\sigma|}}, ~~~~\theta>0,
\end{equation}
where $|\sigma|$ denotes the number of cycles, and
the constant is $c=(\theta)_n:=\Gamma(\theta+n)/\Gamma(\theta)$.
This family is sometimes called Ewens' permutations since the collection of cycle lengths
is then a random partition distributed according to
 the Ewens sampling formula \cite{ABT, CSP}.
The instance $\theta=1$ corresponds to the  uniform distribution under which all permutations $\sigma\in \Sym_n$
are equally likely.

For random permutation $\Pi_n$ with some fixed distribution let $K_{n,r}$ be the number
of cycles of length $r$
and let $K_n:=|\Pi_n|=\sum_{r=1}^n K_{n,r}$
be the total number of cycles.
We call vector $(K_{n,1},\dots,K_{n,n})$   the {\it cycle partition} of $\Pi_n$.
In terms of the cycle partition the order of $\Pi_n$ is the random variable  defined as
\begin{equation}\label{o_def}
O_n:={\rm l.c.m.}\{r\in [n]: K_{n,r}>0\}.
\end{equation}
In a seminal 1967 paper \cite{ET67} Erd\H{o}s and Tur\'an showed
that for the uniform permutation the distribution of $\log O_n$ is
asymptotically normal. Arratia and Tavar\'e \cite{AT92b} extended
this result to Ewens' permutations, by showing that
\begin{equation}\label{et}
{\log O_n-(\theta/2)\log^2 n\over \sqrt{(\theta/3)\log^3 n}} \
\dod \ \mathcal{N}(0,1), \ \ ~~~~n\to\infty.
\end{equation}
The proof in \cite{AT92b} (see also \cite{ABT}, Theorem 5.15),
apparently the shortest one known, is based on the Feller coupling
and asymptotic independence of the $K_{n,r}$'s.

In this paper we generalize the  Erd\H{o}s-Tur\'an law to a much
richer family of random permutations derived from stick-breaking
partitions of the unit interval by means of a simple occupancy
scheme called Kingman's `paintbox process' \cite{CSP}. A toolbox
of methods suitable for the study of Ewens' permutations is no
longer applicable in the wider setting due to the lack of
asymptotic independence of the $K_{n,r}$'s. Instead, extending the
line initiated in \cite{Gne, GIM, GIM2, GINR, GIR, Iks11, Iks111},
we apply   the methods of renewal theory to obtain results on the
weak convergence of the decisive quantity $\log T_n:=\sum_r
K_{n,r}\log r$ which approximates the logarithm of the order of
permutation.  We show that the normal and other stable
distributions can appear as limit laws, as determined by
properties of the  stick-breaking factor.

There have been many studies of random permutations
that
are  conditionally uniform given the value of some permutation statistic
\cite{Diaconis, DiscrMath, GO, Ueltschi}.
Our motivation to consider the class of stick-breaking models has several sources, among which are the
theory of regenerative composition structures \cite{RCS}, more general
exchangeable partitions \cite{CSP} and  the logarithmic combinatorial structures \cite{ABT}.
The present paper  is the first study of a {\it separable
statistic}  $\sum_r K_{n,r} h(r)$ with unbounded function $h$ for
the partitions of integers derived from the general stick-breaking.
It would be interesting to further study separable statistics and approximations to $O_n$
for other permutation models associated with exchangeable partitions.

The organization of the rest of the paper is as follows.
In Section \ref{Perm} we
introduce the class of permutations derived from the stick-breaking.
The principal results are formulated in Section \ref{Main}.
In Section \ref{approx}
we prove that under various regularity conditions
$\log T_n$ yields a good approximation to $\log O_n$ with an error term of the order
$o(\log^{3/2} n)$.
In Section \ref{logpn_asympt} we investigate the weak convergence of $\log T_n$
and prove Theorem \ref{main2}; the method here exploits a link
between the $K_{n,r}$'s and certain perturbed random walks.
Theorem \ref{main} which is our generalization of the
Erd\H{o}s-Tur\'an law follows then as a corollary. The auxiliary
results used in the proofs are collected in the Appendix.

\section{Permutations derived from stick-breaking}\label{Perm}

\paragraph{The Basic Construction}
Let $W$ be a random variable, called {\it stick-breaking factor}, with
values in $(0,1)$.
Consider a multiplicative
renewal point process $\cal Q$ with atoms
$$
Q_0:=1, \ \ Q_j:=\prod_{i=1}^j W_i, \ \ j\in\mn,
$$
where $W_i$  are independent replicas of $W$. The gaps in $\cal
Q$ yield a partition of $[0,1]$ in infinitely many intervals
$(Q_{j+1},Q_{j}]$ accumulating near $0$. Let $U_1,\ldots, U_n$ be
a sample from the uniform $[0,1]$ distribution, independent of
$\cal Q$.
A random permutation $\Pi_n$ is defined by organizing integers $i_1,\dots,i_\ell$ in a cycle $(i_1~\dots~i_\ell)$ if
the following occur:
\begin{itemize}
\item[(i)]   $U_{i_1}<\dots<U_{i_\ell}$,
\item[(ii)] the sample points $U_{i_1},\dots,U_{i_\ell}$ fall in the same interval  $(Q_{j+1},Q_{j}]$ ,
\item[(iii)] only $U_{i_1},\dots,U_{i_\ell}$ out of $U_1,\dots,U_n$ fall in  this interval $(Q_{j+1},Q_{j}]$.
\end{itemize}
Listing the sample points in increasing order and inserting a $|$
between two neighbouring order statistics if they belong to
distinct component intervals of $(0,1]\setminus{\cal Q}$, the
cycle notation of $\Pi_n$ is read left-to-right. \vskip0.2cm For
instance, the list $U_7\,|\,U_3~U_4~ U_2~U_5\,|\,U_6~U_1$ yields
permutation $(7)(3~4~2~5)(6~1)$. To pass to the standard cycle
notation $(1~6)(2~5~3~4)(7)$ one needs to re-arrange the cycles in
the order of increase of their minimal elements, and to rotate
each cycle so that the least element of the cycle appears first.
We prefer, however, to write the cycles and the elements within
the cycles in accord with the natural order on reals, as dictated
by the Basic Construction. A reason for this ordering of cycles is
the following recurrence property:
\begin{itemize}
\item {\it Regeneration}: for  $m\in \{1,\dots,n-1\}$, conditionally given the last cycle of $\Pi_n$ has length $m$,
the cycle partition of $\Pi_n$ with the last cycle deleted has the same distribution as the cycle partition of $\Pi_{n-m}$.
\end{itemize}
It is straightforward from the construction that $\Pi_n$ also satisfies:
\begin{itemize}
\item {\it Coherence}: permutations $\Pi_n$ are defined consistently for all values of $n$.
Passing from $\Pi_{n+1}$ to $\Pi_n$ amounts to removing integer $n+1$ from a cycle.
\item {\it Exchangeability}: the distribution of $\Pi_n$ is invariant under conjugations in ${\frak S}_n$.
Equivalently, given the cycle partition $(K_{n,1},\dots,K_{n,n})$ the distribution of $\Pi_n$ is uniform.
\end{itemize}
In combination with exchangeability, the regeneration property can be re-stated as follows:
given the last cycle of $\Pi_n$ is of length $m$, a permutation resulting from deletion of the last cycle and re-labeling the remaining
elements by the increasing bijection with $[n-m]$ is a distributional copy of $\Pi_{n-m}$.

There are two further useful ways to generate the cycle partition of $\Pi_n$.

\paragraph{A Markov chain representation} Consider a decreasing Markov chain on nonnegative integers with absorption at $0$ and
the {\it decrement matrix}
\begin{equation}\label{dec-ma}
q(n,m)={n\choose m} \frac{{\mathbb E}[W^{n-m}(1-W)^m]}{1-{\mathbb E}W^n}, ~~~~1\leq m\leq n,
\end{equation}
specifying transition probabilities from $n$ to $n-m$.
 For the Markov chain $M_n$ starting at $n$, $K_{n,r}$ is the number of jumps of size $r$ on the path of
$M_n$ from $n$ to $0$. The arrangement of the cycle lengths in the Basic Construction corresponds to the decrements of $M_n$ written in the
time-reversed order.
\paragraph{The infinite occupancy scheme} This model is sometimes called the {\it Bernoulli sieve}
\cite{Gne, GIM, GIM2, GINR, GIR, Iks11, Iks111}. Think of the gaps
$(Q_{j},Q_{j-1}]$ as boxes $1,2,\dots$ with frequencies
\begin{equation}\label{17}
P_j:=W_1W_2\cdots W_{j-1}(1-W_j), \ \ j\in\mn.
\end{equation}
Given the frequencies, balls $1,2,\dots$ are thrown independently so that each ball hits box $j$ with probability $P_j$.
Then $K_{n,r}$ is the number of boxes occupied by exactly $r$ out of the first $n$ balls.
\paragraph{Additive renewal process representation} Mapping $(0,1]$ to ${\mathbb R}_+$ via $x\mapsto -\log x$
sends $\cal Q$ to the additive renewal process
with the generic increment $-\log W$, and sends the uniform sample to a sample from the standard exponential distribution.
The construction of permutation and the occupancy scheme are obviously re-stated in the new variables.

\vskip0.2cm
It has been observed (see \cite{Gne2}, Theorem 2.1) that the
instance of Ewens' permutation fits in the Basic Construction by choosing a
  factor $W\stackrel{d}{=}{\rm beta}(\theta,1)$, with the density
 $$\mmp\{W\in {\rm d}x\}=\theta x^{\theta-1}{\rm
d}x,~~~~x\in (0,1).$$
A better known connection of Ewens' $\Pi_n$ to the stick-breaking stems from the fact that
the scaled by $n$ lengths of the cycles in the normalized notation converge as $n\to\infty$ to $(P_1,P_2,\dots)$ as in (\ref{17})
with $W_j\stackrel{d}{=}{\rm beta}(\theta,1)$.
The distribution of the limit is known as the GEM$(\theta)$ law, which is related to the Poisson-Dirichlet PD$(\theta)$-distribution
through a size-biased permutation of the terms.
As a finite-$n$ counterpart of this dual role of the stick-breaking,
the sequence of lengths of cycles ordered by increase of the minimal elements
and the reversed sequence of the cycle lengths derived from the Basic Construction have the same distribution.
In particular, both sequences can be identified with the sequence of decrements of the
Markov chain $M_n$ with
decrement matrix
$$q(n,n-m)={n\choose m}\frac{(\theta)_{n-m}m!}{(\theta+1)_{n-1}n}.$$
It follows from a result of Kingman
that the coincidence of distributions of the two different arrangements of the unordered  set of the cycle-lengths
characterizes the Ewens permutation within the family of random permutations with
the regenerative property, see \cite{GHP} for this fact and variations.

We note in passing that by a version of the Basic Construction
each system of coherent random permutations $(\Pi_n)_{n\in{\mathbb
N}}$ with the properties of exchangeability and regeneration, with
respect to deletion of a cycle of $\Pi_n$ chosen by some random
rule, uniquely corresponds to a random regenerative subset of
${\mathbb R}_+$ which coincides with the closed range
of a subordinator $S$ \cite{RCS} . Distinguishing features of the subfamily in
focus in the present paper are: (1) $S$ is a compound Poisson
process with jumps distributed like $|\log W|$; (2) the last cycle
of $\Pi_n$ has the length of the order $O(n)$
as $n$ grows.

\section{Main results}\label{Main}

In the sequel we use the following notation for the moments
of the stick-breaking factor
$$\mu:=\me |\log W|, \ \ \sigma^2:={\rm Var}\,(\log W) \ \ \text{and} \ \ \nu:=\me |\log
(1-W)|,$$ which may be finite or infinite. We shall also use the
notion of slow variation. Function $\ell:(0,\infty)\to (0,\infty)$
is called {\it slowly varying at $\infty$} if for all $\lambda>0$,
$$
\lim_{x\to\infty}\frac{\ell(\lambda x)}{\ell(x)}=1.
$$
Our purpose is to extend \eqref{et} to a wider class of random
permutations $\Pi_n$  derived from stick-breaking,  along the following lines.
\begin{thm}\label{main}
Suppose the law of $W$ is absolutely continuous with a density
$f$.
\begin{enumerate}
\item[{\rm I.}] If there exist $\delta_1\geq 0$ and $\delta_2\geq 0$ such that $f$ is nonincreasing on $(0,\delta_1)$,
bounded on $[\delta_1,1-\delta_2]$ and nondecreasing on
$(1-\delta_2,1)$ then
\begin{itemize}
\item[\rm (a)] If $\sigma^2<\infty$ then, with
\begin{equation}\label{26}
b_n=\mu^{-1}\bigg(2^{-1}\log^2 n-\int_0^{\log n}\int_0^z
\mmp\{|\log (1-W)|>x\}{\rm d}x{\rm d}z\bigg)
\end{equation}
and $a_n=((3\mu^3)^{-1}\sigma^2\log^3 n)^{1/2}$, the limiting
distribution of $(\log O_n-b_n)/a_n$ is standard normal.

\item[\rm(b)] If $\sigma^2=\infty$, and
$$\int_0^x y^2\, \mmp\{|\log W|\in {\rm d}y\} \ \sim \ \ell(x), \ \ x\to \infty,$$ for some $\ell$ slowly varying at $\infty$,
then, with $b_n$ given by \eqref{26} and
$$a_n=(3\mu^3)^{-1/2}
c_{[\log n]}\log n,$$
where $(c_n)$ is any positive sequence
satisfying $\lin\,n\ell(c_n)/c_n^2=1$, the limiting distribution
of $(\log O_n-b_n)/a_n$ is standard normal.
\item[\rm(c)] If
\begin{equation}\label{domain1}
\mmp\{|\log W|>x \} \ \sim \ x^{-\alpha}\ell(x), \ \ x\to \infty,
\end{equation}
for some $\ell$ slowly varying at $\infty$ and $\alpha\in (1,2)$
then, with $b_n$ as in \eqref{26} and
$$a_n=((\alpha+1)\mu^{\alpha+1})^{-1/\alpha} c_{\lfloor\log n\rfloor}\log n,$$
where $(c_n)$ is any positive sequence satisfying $\lin
\,n\ell(c_n)/c_n^\alpha=1$, the limiting distribution of $(\log
O_n-b_n)/a_n$ is the $\alpha$-stable law with characteristic
function
\begin{equation}\label{st}
u\mapsto \exp\{-|u|^\alpha
\Gamma(1-\alpha)(\cos(\pi\alpha/2)+i\sin(\pi\alpha/2)\, {\rm
sgn}(u))\}, \ u\in\mr.
\end{equation}
\end{itemize}

\item[{\rm II.}] If for some $\alpha\in[0,1)$
\begin{equation}\label{density_cond}
\sup_{x\in[0,1]}x^\alpha(1-x)^{\alpha}f(x)<\infty;
\end{equation}
then $\sigma^2<\infty$ and
$$
{\log O_n - (2\mu)^{-1} \log^2 n\over
\sqrt{(3\mu^3)^{-1}\sigma^2\log^3 n}} \ \dod \ \mathcal{N}(0,1), \
\ n\to\infty.
$$
\end{enumerate}
\end{thm}

In particular, these conditions cover all bounded densities, and
all {\rm beta}$(a,b)$ densities with arbitrary parameters $a,b>0$.
Following an approach exploited by previous authors we derive our
extension of the Erd\H{o}s-Tur\'an law in two steps. We first show
that the accompanying quantity $\log T_n$ yields a good
approximation to $\log O_n$, where
\begin{equation}\label{p_n_def}
T_n:=\prod_{r=1}^n r^{K_{n,r}},
\end{equation}
is the product of cycle lengths of $\Pi_n$.
Then we study the weak convergence of $\log T_n$.


Functional $\log T_n$ is an instance of a {\it separable
statistic} of the form  $\sum_r K_{n,r}h(r)$ (the terminology is
borrowed from \cite{Med, Mirah}, where it was used in the context of
occupancy problems). Functionals $K_{n,r}$ and $K_n$ are
themselves of this kind with some indicator functions $h$,
but for $\log T_n$ the function $h$ is unbounded. For Ewens'
permutations quite general separable statistics were studied by
Babu and Manstavi\v{c}ius, see e.g. \cite{Man1, Man2}. 


\begin{thm}\label{main2} If $W$ satisfies the moment conditions required, respectively, in
parts {\rm (a)}, {\rm (b)} and {\rm (c)}   of {\rm ~Theorem
\ref{main}}, then the conclusions of parts {\rm (a)}, {\rm (b)}
and {\rm (c)} hold with $\log O_n$ replaced by $\log T_n$, without
the assumption regarding the existence of density of $W$.
\end{thm}


\paragraph{Example: beta distributions}
Assuming $W\stackrel{d}{=}{\rm beta}(\theta, 1)$ we have
$\mu=\theta^{-1}$, $\sigma^2=\theta^{-2}$ and
$$\lin {\int_0^{\log n}\int_0^z \mmp\{|\log (1-W)|>x\}{\rm d}x{\rm d}z\over
\log^{3/2}n}=0,$$
since the numerator is $O(\log n)$.
Application of Theorem \ref{main2} (a) yields
$${\log T_n-(\theta/2)\log^2 n\over \sqrt{(\theta/3)\log^3 n}} \ \dod \ \mathcal{N}(0,1), \ \ n\to\infty,$$
which was previously obtained in  \cite{AT92b}, equation (34).

\section{Approximation of $\log O_n$ by $\log T_n$}\label{approx}

For $j\in [n]$
set $$D_{n,j}:=\sum_{r\leq
n,\,j|r}K_{n,r}=\sum_{r=1}^{\lfloor\,n/j\rfloor\,}K_{n,\,rj}.$$ For the later use we
need appropriate bounds for the expectation
$\me(D_{n,j}-1)^+$.
\begin{lemma}\label{dnk_mom}
Under the assumptions of {\rm ~Theorem \ref{main}} the asymptotic relations
\begin{eqnarray}\label{assert1}
\me (D_{n,j}-1)^{+}=O\left(\frac{\log n}{j}\right),\\
\label{assert2}
\me (D_{n, j}-1)^{+}=O\left(\frac{\log^2 n}{j^2}\right)
\end{eqnarray}
hold  uniformly in  $j\in[n]$.

\begin{proof}
Define $D^{(1)}_{n,j}:=\sum_{r=1}^{\lfloor
n/j\rfloor\,-1}K_{n,rj}$. It is obvious that $ (D_{n,j}-1)^{+}\leq
D^{(1)}_{n,j}$.

Let $A_n$ be  the length of the last cycle of $\Pi_n$, with
distribution $\mmp\{A_n=j\}=q(n,j)$ as in (\ref{dec-ma}). One can
check that the bivariate array $(D^{(1)}_{n,j})$
satisfies the  distributional recurrence
\begin{eqnarray}\label{rec_d1nk}
D^{(1)}_{n,j}&=&0,\;\;n<j,\nonumber\\
D^{(1)}_{n,j}&\od & 1_{\{j | A_n,\,j\leq A_n\leq n-j
\}}+\hat{D}^{(1)}_{n-A_n, j},\;\;n\geq j,
\end{eqnarray}
where the variables $\hat{D}^{(1)}_{n,k}$ are
 assumed independent of
$\Pi_n$ and marginally distributed like $D^{(1)}_{n,k}$ for all
$n,k\in\mn$. Taking expectations yields
$$
\me D^{(1)}_{n,j}=\sum_{r=1}^{\lfloor
\,n/j\rfloor\,-1}\mmp\{A_n=rj\}+\sum_{i=j}^{n}\mmp\{n-A_n=i\}\me
D^{(1)}_{i,j}\;\;~~{\rm for~} n\geq j,
$$
and $\me D^{(1)}_{n,j}=0$ for $n<j$.

By Lemma \ref{divis_bin}
\begin{equation}\label{inhomogeneous_term1}
j\sum_{r=1}^{\lfloor n/j\rfloor -1}\mmp\{A_n=rj\}=O(1),\;\;j\leq
n,\;\;n\in\mn.
\end{equation}
Now relation \eqref{assert1} follows by the virtue of part (i) of
Lemma \ref{gne_lemma} and Lemma \ref{recursion_l} with $c_j=j$.

To prove the second assertion \eqref{assert2}, note that
\begin{eqnarray*}
(D_{n,j}-1)^{+}&=&(D_{n,j}-1)^{+}1_{\{K_{n,\lfloor n/j\rfloor j}=0\}}=(D^{(1)}_{n,j}-1)^{+}1_{\{K_{n,\lfloor n/j\rfloor j}=0\}}\\
&\leq&(D^{(1)}_{n,j}-1)^{+}\leq
D^{(1)}_{n,j}(D^{(1)}_{n,j}-1)/2=:D^{(2)}_{n,j},
\end{eqnarray*}
holds almost surely. Squaring relation \eqref{rec_d1nk} and using
\eqref{inhomogeneous_term1} and \eqref{assert1} yield
$$
\me D^{(2)}_{n,j}= O(j^{-2}\log
n)+\sum_{i=j}^{n}\mmp\{n-A_n=i\}\me D^{(2)}_{i,j},\;\;n\geq j, \
j\in\mn,
$$
Finally, application of part (ii) of Lemma \ref{gne_lemma} and Lemma \ref{recursion_l} with $c_j=j^2$ establish \eqref{assert2}, as wanted.
\end{proof}
\end{lemma} 
%
The following estimate of the difference $\log T_n-\log O_n$
generalizes Lemma 4 in \cite{AT92b}.
\begin{lemma}\label{approx_lemma}
Under the assumptions of {\rm ~Theorem \ref{main}} the following asymptotic relations hold
$$\me \left(\log T_n-\log O_n\right)=O\left(\log n \log\log n\right), \ \ n\to\infty.$$
\begin{proof}
We start with a known representation (p.~289 in \cite{Pittel})
$$
\log T_n - \log O_n = \sum_{p\in\mathcal{P}}\log p\sum_{s\geq 1}
(D_{n,\,p^s}-1)^+,
$$
where $\mathcal{P}$ denotes the set of prime numbers, which
implies
\begin{eqnarray*}
\me \left(\log T_n-\log O_n\right)&=& \sum_{p\in\mathcal{P},\,s\geq 1}\log p \,\me (D_{n,\,p^s}-1)^{+}\\
 &\leq& \sum_{p\in\mathcal{P},\, p\leq\log n} \log p\, \me (D_{n,p}-1)^{+}+\sum_{p\in\mathcal{P},\,s\geq 2,\, p^s\leq\log n} \log p\,
 \me (D_{n,p^s}-1)^{+}\\
 &+& \sum_{j>\log n}\log j\, \me (D_{n,j}-1)^{+}=:S_1(n)+S_2(n)+S_3(n).
\end{eqnarray*}
Applying \eqref{assert1} along with Theorem 4.10 in \cite{Apostol} which
states that
$$\sum_{p\in\mathcal{P},\, p\leq x}{\log p\over p} = \log
x + O(1), \ \ x\to\infty,$$ proves $S_1(n)=O(\log n\log\log n)$. Using
\eqref{assert1} again yields $S_2(n)=O(\log n)$. Finally, from
\eqref{assert2} we infer $S_3(n)=O(\log n\log\log n)$. The proof is
complete.
\end{proof}
\end{lemma}

\section{Weak convergence of $\log T_n$}\label{logpn_asympt}

To prove Theorem \ref{main2} we shall exploit a  strategy as in
\cite{GIM} (see also \cite{Iks11}), which amounts to connecting
the asymptotics of $\log T_n$ (as $n\to\infty$) with that of the
`small frequencies' $P_k$ (as $k\to\infty$). Since the process
$(\log P_k)_{k\in\mn}$ defined by \eqref{17} is a particular {\it
perturbed random walk}, we start in Subsection \ref{prw} with
developing necessary backgrounds on the perturbed random walks.
These results are further specialized to $\log P_k$ in Subsection
\ref{pr}, which eventually allows to complete the proof of Theorem
\ref{main2}.

\subsection{Results for perturbed random walks}\label{prw}

Let $(\xi_k, \eta_k)_{k\in\mn}$ be independent copies of a random
vector $(\xi, \eta)$ with arbitrarily dependent components $\xi>
0$ and $\eta\geq 0$. We assume that the law of $\xi$ is
nondegenerate and that the law of $\eta$ is not the Dirac mass at $0$.
Set $F(x):=\mmp\{\eta\leq x\}$ and $r(x):=\int_0^x(1-F(y)){\rm
d}y$.

For $(S_k)_{k\in\mn_0}$ a random walk with $S_0=0$ and increments
$\xi_k$, the sequence $(T_k)_{k\in\mn}$ with
$$T_k:=S_{k-1}+\eta_k, \ \ k\in\mn,$$ is called a {\it perturbed random
walk}. Since $\underset{k\to\infty}{\lim} T_k =\infty$ a.s., there
is some finite number
$$N(x):=\#\{k\in\mn: T_k\leq x\}, \ \ x\geq 0,$$
of sites visited on the interval $[0,x]$. Set also
$$\rho(x):=\#\{k\in\mn_0: S_k\leq x\}\ = \ \inf\{k\in\mn: S_k>x\},
\ \ x\geq 0,$$ and $$M(x):=\sum_{k\geq 0}\me \left(1_{\{T_{k+1}\leq
x\}}\big|S_k\right)=\sum_{k\geq 0}F(x-S_k), \ \ x\geq 0.$$

The main result of this subsection is given next.
\begin{thm}\label{main1}
Assume that ${\tt m}:=\me \xi<\infty$ and $${\rho(x)-{\tt
m}^{-1}x\over c(x)} \ \dod \ Z, \ \ x\to\infty.$$ Then
$$I(x):={\int_0^x (N(y)-{\tt m}^{-1}(y-r(y))){\rm d}y\over xc(x)} \
\dod \ \int_0^1 Z(y){\rm d}y=:X, \ \ x\to\infty,$$ where
$(Z(t))_{t\geq 0}$ is a stable L\'{e}vy process such that $Z(1)$
has the same law as $Z$.
\end{thm}
\begin{rem}
It is known (see Proposition 27 in \cite{Neg}) that $c(x)\ \sim \
x^\beta \ell_1(x)$ for some $\beta\in [1/2,1)$ and some slowly
varying $\ell_1$, where $\beta$ and $\ell_1$ depend on the
distribution of $\xi$. Furthermore, if $\beta=1/2$ then either
$\ell_1(x)={\rm const}$ or $\lix \ell_1(x)=\infty$. Thus, in any
case,
\begin{equation}\label{18}
{x\over c^2(x)}=O(1), \ \ x\to\infty.
\end{equation}
\end{rem}

The proof of Theorem \ref{main1} relies heavily upon the
following.
\begin{lemma}\label{aux1}
Under the assumption and notation of{\rm~ Theorem \ref{main1}},
\begin{equation}\label{15}
J(x):={\int_0^x (\rho(y)-{\tt m}^{-1}y){\rm d}y\over xc(x)} \ \dod
\ X, \ \ x\to\infty.
\end{equation}
\end{lemma}
\begin{proof}
It is known (see Theorem 1b in \cite{Bing73}) that
\begin{equation}\label{16}
W_x(\cdot):={\rho(x\cdot)-{\tt m}^{-1}(x\cdot)\over c(x)}
\Rightarrow \ Z(\cdot), \ \ x\to\infty,
\end{equation}
in $D[0,\infty)$ in the $M_1$-topology. Since integration is a
continuous
operator from $D[0,\infty)$ to $D[0,\infty)$, we have
$$\int_0^1 W_x(y){\rm d}y\ \dod \ \int_0^1 Z(y){\rm d}y, \ \ x\to\infty,$$
which is equivalent to \eqref{15}.
\end{proof}

\paragraph{Remark} When $Z(\cdot)$ is a Brownian motion, the one-dimensional
convergence in \eqref{15} can be upgraded to the functional limit
theorem. Indeed, since $(Z(t))$ is continuous the convergence in
\eqref{16} is equivalent to the locally uniform convergence.
Furthermore, the integration
$z(\cdot)\ \mapsto \ \int_0^{(\cdot)}z(y){\rm d}y$
is continuous w.r.t.\, the locally uniform convergence. Hence, by
the continuous mapping theorem,
$$\int_0^{(\cdot)}W_x(y){\rm d}y\ \Rightarrow \ \int_0^{(\cdot)}Z(y){\rm d}y, \ \ x\to\infty$$
in $D[0,\infty)$.
\vskip0.2cm
Lemma \ref{aux} collects some facts borrowed from
\cite{GIM}.
\begin{lemma}\label{aux}
{\rm (a)} $\me (N(x)-M(x))^2=o(x)$, as $x\to\infty$.\newline
{\rm (b)} Under
the assumption and notation of {\rm ~Theorem \ref{main1}},
$${\underset{y\in [0,x]}{\sup}\,(\rho(y)-{\tt m}^{-1}y )\over
c(x)} \ \dod \ \underset{t\in [0,1]}{\sup}\,Z(t), \ \
{\rm as~~}
x\to\infty,$$ and
$${\underset{y\in [0,x]}{\inf}\,(\rho(y)-{\tt
m}^{-1}y )\over c(x)} \ \dod \ \underset{t\in [0,1]}{\inf}\,Z(t),
\ \
{\rm as~~}x\to\infty.$$
\end{lemma}
\begin{myproof} {\it of} {\rm~ Theorem \ref{main1}}. Applying the Cauchy-Schwarz inequality,
$${\me \bigg(\int_0^x
|N(y)-M(y)|{\rm d}y \bigg)^2\over x^2c^2(x)}\leq {\int_0^x\me
(N(y)-M(y))^2{\rm d}y \over xc^2(x)}={o(x^2)\over x^2}{x\over
c^2(x)},$$
where for the final estimate Lemma \ref{aux}(a) was utilized.
In view of \eqref{18}, the latter expression goes
to $0$,
which implies that
\begin{equation}\label{2}
{\int_0^x (N(y)-M(y)){\rm d}y \over xc(x)}\ \tp \ 0, \ \
x\to\infty.
\end{equation}
Since
$${\int_0^x (N(y)-{\tt m}^{-1}(y-r(y))){\rm d}y\over
xc(x)}= {\int_0^x (N(y)-M(y)){\rm d}y \over xc(x)}+{\int_0^x
(M(y)-{\tt m}^{-1}(y-r(y))){\rm d}y \over xc(x)},$$ we have to
prove that the second summand converges in distribution to $X$.

With $\delta\in (0,1)$ such that $y^\delta=o(c(y))$, write for
$y>1$
\begin{eqnarray*} F(y)+M(y)-{\tt m}^{-1}(y-r(y))&=& \int_0^y
(\rho(y-z)-{\tt m}^{-1}(y-z)){\rm
d}F(z)\\&=&\int_0^{y^\delta}\ldots+\int_{y^\delta}^y\ldots\\&=&T_1(y)+T_2(y).
\end{eqnarray*}
In view of $$T_1(y)\leq (\rho(y)-{\tt
m}^{-1}y)F(y^\delta)+{\tt m}^{-1}y^\delta F(y^\delta)\leq (\rho(y)-{\tt
m}^{-1}y)+{\tt m}^{-1}y^\delta
$$
we have $${\int_0^x T_1(y){\rm d}y\over xc(x)}\leq {\int_0^1
T_1(y){\rm d}y\over xc(x)}+{\int_0^x (\rho(y)-{\tt m}^{-1}y){\rm
d}y\over xc(x)}+{(\delta+1)^{-1} x^\delta\over {\tt m}c(x)} \ \dod \
0+X+0=X,$$ where the last step is justified by Lemma \ref{aux1}
and the choice of $\delta$. Further, $$T_1(y)\geq (\rho(y)-{\tt
m}^{-1}y)-({\rho(y)-{\tt
m}^{-1}y})(1-F(y^\delta))-(\rho(y)-\rho(y-y^\delta)).$$ Since
$${\me \int_1^x (\rho(y)-\rho(y-y^\delta)){\rm d}y\over xc(x)}\leq
{\int_1^x \me \rho(y^\delta){\rm d}y\over xc(x)}\leq
{\me\rho(x^\delta)\over x^\delta}{x^\delta\over c(x)} \ \to \ {\tt
m}^{-1}\cdot 0=0,$$ by the elementary renewal theorem and the
choice of $\delta$, we conclude that $${\int_1^x
(\rho(y)-\rho(y-y^\delta)){\rm d}y\over xc(x)}\ \tp \ 0.$$
Therefore,
\begin{eqnarray*}
{\int_0^x T_1(y){\rm d}y\over x c(x)}&\geq& {\int_1^x
(\rho(y)-{\tt m}^{-1}y){\rm d}y\over xc(x)}-{\underset{0\leq y\leq
x}{\sup}\,(\rho(y)-{\tt m}^{-1}y)\over c(x)}{\int_1^x
(1-F(y^\delta)){\rm d}y\over x}\\&-&{\int_1^x
(\rho(y)-\rho(y-y^\delta)){\rm d}y\over xc(x)}\\&\dod& X-0-0=X,
\end{eqnarray*}
by Lemma \ref{aux1} and Lemma \ref{aux} (b).

Finally,
$$\underset{0\leq z\leq y}{\inf}\,(\rho(z)-{\tt
m}^{-1}z) (F(y)-F(y^\delta))\leq T_2(y)\leq \underset{0\leq z\leq
y}{\sup}\,(\rho(z)-{\tt m}^{-1}z) (F(y)-F(y^\delta))$$ entails
$${\int_0^x T_2(y){\rm d}y\over xc(x)}\leq {\int_0^1 T_2(y){\rm d}y\over xc(x)}+{\underset{0\leq z\leq
x}{\sup}\,(\rho(z)-{\tt m}^{-1}z)\over
c(x)}{\int_1^x(F(y)-F(y^\delta)){\rm d}y\over x} \ \tp \ 0,$$
where the last step follows from Lemma \ref{aux}(b) and the
trivial fact that the last ratio goes to $0$ for any distribution
function $F$. Similarly, $${\int_0^x T_2(y){\rm d}y\over
xc(x)}\geq {\int_0^1 T_2(y){\rm d}y\over
xc(x)}+{\underset{0\leq z\leq x}{\inf}\,(\rho(z)-{\tt
m}^{-1}z)\over c(x)}{\int_1^x(F(y)-F(y^\delta)){\rm d}y\over x} \
\tp \ 0,$$ Putting the pieces together completes the proof.~~~~~~~~~~~~~~~~~~~~~~~~~~~~~~~~~~~~~~~~~~~~~~~~~~~~~~~~~~~~\end{myproof}

\subsection{Proof of Theorem \ref{main2} and Theorem \ref{main}}\label{pr}

{\it Proof of Theorem \ref{main2}}. We shall make use of the
Poissonized version of the occupancy model with random frequencies
$(P_k)$, in which balls are thrown in boxes at epochs of a unit
rate Poisson process $(\pi_t)_{t\geq 0}$. For simplicity we use
notation $V(t)=\log T_{\pi_t}$.

Set
\begin{eqnarray*}\label{rhox} \rho^\ast(x):=\inf \{k\in\mn:
W_1\dots W_k< e^{-x}\}, \ \ x\geq 0,
\end{eqnarray*}
and
\begin{eqnarray*}\label{nx}
N^\ast(x)&:=&\#\{k\in\mn: P_k\geq e^{-x}\}\\&=&\#\{k\in\mn:
W_1\cdots W_{k-1}(1-W_k)\geq e^{-x}\},~~~  x\geq 0.
\end{eqnarray*}

First of all, we need a refined large deviation result for
$(\pi_t)$ itself: for $t>1$,
\begin{equation}\label{ld}
\mmp\{\pi_t\leq (1-\varepsilon_t)t\}\leq
\exp(-t(\varepsilon_t+\log(1-\varepsilon_t)(1-\varepsilon_t)))=:q(t),
\end{equation}
where $\varepsilon_t:=t^{-\beta}$, for any $\beta\in (0,1/2)$.
Note that $\lit q(t)=0$ with $(-\log q(t))\sim t^{1-2\beta}$.
Inequality \eqref{ld} is the Chernoff bound for the Poisson distribution and
follows in a standard way by first applying
Markov's inequality to $e^{-s\pi_t}$ and then minimizing the
right-hand side over $s$.

For $j=1,2$, set $$f_j(t):=\me (\log^+ \pi_t)^j =e^{-t}\sum_{k\geq
2}\log^j k (t^k/ k!), \ \ t\geq 0.$$ These functions are
nondecreasing and differentiable with $f_j(0)=0$ and
\begin{equation}\label{0}
f^\prime_j(0)=0.
\end{equation}
Let us prove that
\begin{equation}\label{li}
\lit (f_1(t)-\log t)=0
\end{equation}
and
\begin{equation}\label{li2}
\lit h(t)=0,
\end{equation}
where $h(t):={\rm Var}(\log^+\pi_t)$. To this end, write
\begin{eqnarray}\label{28}
f_1(t)-\log t \leq \me \log (\pi_t+1)-\log t \leq \log(t+1)-\log
t\leq t^{-1},
\end{eqnarray}
where at the second step Jensen's inequality has been utilized.
Similarly,
\begin{eqnarray}\label{27}
f_2(t)-\log^2 t &\leq& \me \log^2 (\pi_t+1)-\log^2 t\nonumber\\
&\leq& \log^2 (t+1)-\log^2 t\nonumber\\ &\leq& 2t^{-1}\log(t+1).
\end{eqnarray}
Note that we actually work on the set $\{\pi_t\geq 2\}$ and that
the function $t\mapsto \log^2(1+t)$ is concave for $t\geq 2$.

Furthermore, for large enough $t$, and $\varepsilon_t$ as defined
above,
\begin{eqnarray*}
f_1(t)-\log t&\geq& \me (\log^+\pi_t-\log
t)1_{\{\pi_t>(1-\varepsilon_t)t\}}-\log t\mmp\{\pi_t\leq
(1-\varepsilon_t)t\}\\&\geq& \log
(1-\varepsilon_t)\mmp\{\pi_t>(1-\varepsilon_t)t\}-q(t)\log
t=:p(t).
\end{eqnarray*}
and the last expression goes to zero (with rate $t^{-\beta}$), as
$t\to\infty$. Combining this inequality with \eqref{28} proves
\eqref{li}. Note also that
\begin{eqnarray}\label{25}
f_1^2(t)&=&\log^2t+2\log t(f_1(t)-\log t)+(f_1(t)-\log
t)^2\nonumber\\&\geq& \log^2t+2p(t)\log t.
\end{eqnarray}
Hence $$h(t)=f_2(t)-f_1^2(t)\overset{\eqref{27},\eqref{25}}{\leq}
2(t^{-1}\log(t+1)-p(t)\log t)=O(\log t/t^\beta),$$ which proves
\eqref{li2}\footnote{Alternatively, both \eqref{li} and
\eqref{li2} can be deduced from Theorem 4 in \cite{Szpankowski}.
To keep the paper self-contained we prefer to give an elementary
real-analytic argument.}.


The basic observations for the subsequent work are given and
proved next:
\begin{eqnarray}\label{bas}
\me (V(t)|(P_k))&=&\sum_{j\geq 1} f_1(tP_j)\nonumber\\&=&
\int_1^\infty f_1(t/x){\rm d}N^\ast(\log
x)\nonumber\\&=&\int_0^{\log t} (\log t-x) {\rm d}N^\ast(x)+O_P(\log
t)\\&=&\int_0^{\log t}N^\ast(x){\rm d}x+O_P(\log t)\nonumber
\end{eqnarray}
and
\begin{eqnarray}\label{bas100}
{\rm Var}\,(V(t)|(P_k))&=&\sum_{j\geq 1} h(tP_j)\nonumber\\&=&
O_P(\log t),
\end{eqnarray}
where $O_P(\log t)$ means that $O_P(\log t)/\log t$ is bounded in
probability.

The a.s.\,finiteness of the conditional expectation (and even its
integrability) can be justified as follows: $$\me \log T_n\leq
(\log^+n)\me K_n\leq n\log^+n.$$ Hence $\me V(t)\leq \me
\pi_t\log^+\pi_t<\infty$. The integrability of the conditional
variance can be checked similarly.

Since $N^\ast(\log y)\leq \rho^\ast(\log y)$, and $\rho^\ast(\log
y)=O_P(\log y)$ we conclude that
\begin{equation}\label{nl}
N^\ast(\log y)=O_P(\log y).
\end{equation}
Using this and \eqref{li} gives
$$\int_1^t f_1(t/x){\rm d}N^\ast(\log x)=\int_0^{\log t}(\log
t-x){\rm d}N^\ast(x)+O_P(\log t).$$ In fact, only boundedness of
$f_1(t)-\log t$ was used. Further,
\begin{eqnarray*}
\int_t^\infty f_1(t/x){\rm d}N^\ast(\log x)&=&-f_1(1)N^\ast(\log
t)+\int_0^1 N^\ast(\log t-\log x)f_1^\prime(x){\rm d}x\\
&\overset{\eqref{nl}}{\leq}& O_P(\log t)+\rho^\ast(\log
t)f_1(1)\\&+&\int_0^1 (\rho^\ast(\log t-\log x)-\rho^\ast(\log
t))f_1^\prime(x){\rm d}x\\&=&O_P(\log t),
\end{eqnarray*}
since by the well-known bound for the renewal function
$$\me \int_0^1 (\rho^\ast(\log t-\log x)-\rho^\ast(\log
t))f_1^\prime(x){\rm d}x \leq \int_0^1 (C_1|\log x|+C_2)
f_1^\prime(x){\rm d}x\overset{\eqref{0}}{<}\infty,$$ where $C_1$
and $C_2$ are positive constants. Thus we have proved \eqref{bas}.
The proof of \eqref{bas100} follows the same pattern, the only
minor difference being that now we use inequality
$$\int_t^\infty h(t/x){\rm d}N^\ast(\log x)\leq \int_t^\infty f_2(t/x){\rm d}N^\ast(\log x)$$
and \eqref{0} for $f_2$.

Throughout the rest of the proof we apply results of Subsection
\ref{prw} to the vector $(\xi, \eta):=(|\log W|, |\log(1-W)|)$.
With this specific choice the quantities $\rho(x)$ and $N(x)$
defined in Subsection \ref{prw} turn into $\rho^\ast(x)$ and
$N^\ast(x)$.

Let $(X(t))_{t\geq 0}$ be a L\'{e}vy process with $\log \me
e^{izX(1)}=\psi(z)$, $z\in\mr$. Then
\begin{equation}\label{Le}
\log \me \exp\bigg(iz\int_0^1 X(t){\rm d}t\bigg)=\int_0^1
\psi(zs){\rm d}s,
\end{equation}
which follows from a Riemann approximation to the integral.

Assume that the assumptions of Theorem \ref{main2} hold which
implies that the assumption of Theorem \ref{main1} (with $\rho$
replaced by $\rho^\ast$) holds. By scaling $Z$ and $c(x)$, if
necessary, we can assume that $Z$ has the standard normal
distribution under the assumptions of parts (a) and (b) of Theorem
\ref{main2} and that $Z$ has a stable law with characteristic
function \eqref{st} under \eqref{domain1}. Then \eqref{Le} implies
that $X\od 3^{-1/2}Z\od \mathcal{N}(0,1/3)$ in the first case, and
that $X\od (\alpha+1)^{-1/\alpha}Z$ in the second case.

By Theorem \ref{main1},
$${\int_0^{\log t} (N^\ast(y)-\mu^{-1}(y-r^\ast(y))){\rm d}y\over
c(\log t)\log t} \ \dod \ X, \ \ t\to\infty,$$ where
$r^\ast(y):=\int_0^y\mmp\{|\log (1-W)|>z\} {\rm d}z$. Since $\lit
c(t)=\infty$, using \eqref{bas} yields $${\me
(V(t)|(P_k))-\mu^{-1}\bigg(2^{-1}\log^2t-\int_0^{\log t}
r^\ast(y){\rm d}y\bigg)\over c(\log t)\log t} \ \dod \ X, \ \
t\to\infty,$$ and hence
$${V(t)-\mu^{-1}\bigg(2^{-1}\log^2t-\int_0^{\log t}
r^\ast(y){\rm d}y\bigg)\over c(\log t)\log t} \ \dod \ X, \ \
t\to\infty,$$ by virtue of \eqref{bas100} and Chebyshev's
inequality.

Now we have to de-Poissonize, i.e., to pass from the Poissonized
occupancy model to the fixed-$n$ model. This is simple as $(\log
T_n)$ is a nondecreasing sequence. Set
$$b(t):=\mu^{-1}\bigg(2^{-1}\log^2 t-\int_0^{\log
t}\mmp\{|\log(1-W)|>y\}{\rm d}y\bigg) \ \ \text{and} \ \
a(t):=c(\log t)\log t.$$ Recall that we take a properly adjusted
$c(x)$. Since $a(t)$ grows faster than the logarithm, we have
$$\lit {b(t)-b(\lfloor\,t(1\pm \varepsilon)\rfloor\,)\over
a(t)}=0,$$ for every $\varepsilon>0$. This together with slow
variation of $a(t)$ give
$$X_{\pm}(t):={V(t)-b(\lfloor t(1\pm \varepsilon)\rfloor)
\over a(\lfloor t(1\pm \varepsilon)\rfloor)}\ \dod \ X.$$ By the
monotonicity of $(\log T_n)$, we have
\begin{eqnarray*} X_+(t)&=&
X_+(t)1_{D_t}+X_+(t)1_{(D_t)^c} \\&\leq&
{V_{\lfloor(1+\varepsilon)t\rfloor}-b(\lfloor t(1+
\varepsilon)\rfloor)\over a(\lfloor t(1+\varepsilon)\rfloor
)}1_{D_t}+X_+(t)1_{(D_t)^c},
\end{eqnarray*}
where $D_t:=\big\{\pi_t\in [\lfloor(1-\varepsilon)t\rfloor,
\lfloor (1+\varepsilon)t\rfloor]\big\}$. Since ${\mathbb P}(D_t)
\to 1$, hence $X_+(t)1_{(D_t)^c}\ \ \tp \ 0,$ we conclude that
$$\mmp\{X>x\} \leq
\underset{n\to\infty}{\lim\inf}\,\mmp\bigg\{{\log T_n-b(n)\over
a(n)}>x \bigg\},$$ for all $x\in\mr$. To prove the converse
inequality for the upper bound one can proceed in a similar
manner.

It remains to set $b_n=b(n)$, and $a_n=(\alpha+1)^{-1/\alpha}a(n)$
if the assumption of part (c) holds, and $a_n=3^{-1/2}a(n)$ if the
assumptions of parts (a) and (b) hold. The fact that the
so-defined $a_n$ and $b_n$ are of the form as stated in Theorem
\ref{main2} follows from considerations above and from, for
instance, Proposition 27 in \cite{Neg}. The proof of Theorem
\ref{main2} is complete.

\noindent {\it Proof of Theorem \ref{main}}. By Theorem
\ref{main2}, $(\log T_n-b_n)/a_n$, with case-dependent $a_n$ and
$b_n$ defined in Theorem \ref{main}, weakly converges. In
particular, we know that $\log^{3/2}n=O(a_n)$. It remains to apply
Lemma \ref{approx_lemma} and Markov's inequality. The proof of
Theorem \ref{main} is complete.

\section{Appendix}
The following lemma is a simple consequence of Proposition 3 in \cite{Gne}.
\begin{lemma}\label{gne_lemma}
Assume that the sequence $a_n$ satisfies the following recurrence relation
$$
a_0=0,\;\;a_n=b_n+\sum_{k=0}^n q(n,k)a_{n-k},\;\;n\in\mn.
$$
Then
\begin{itemize}
\item[(i)] if $b_n=O(1)$ then $a_n=O(\log n)$, as $n\to\infty$,
\item[(ii)] if $b_n=O(\log n)$ then $a_n=O(\log^2 n)$, as $n\to\infty$.
\end{itemize}
\end{lemma}

The next lemma verifies \eqref{inhomogeneous_term1} which is a key
ingredient of the proof of Lemma \ref{dnk_mom}.
\begin{lemma}\label{divis_bin}
Relation \eqref{inhomogeneous_term1} holds provided
the density $f$ of $W$ satisfies any of the following two conditions:
\begin{itemize}
\item[\rm(i)] condition \eqref{density_cond} holds for some $\alpha\in[0,1)$.
\item[\rm(ii)] there exist $\delta_1\geq 0$ and $\delta_2\geq 0$ such that $f$ is nonincreasing on $(0,\delta_1)$,
bounded on $[\delta_1,1-\delta_2]$ and nondecreasing on
$(1-\delta_2,1)$.
\end{itemize}

\begin{proof}
We start with easier part (i). We have
\begin{eqnarray*}
k\sum_{r=1}^{\lfloor n/k\rfloor -1}\mmp\{A_n=rk\}&=&\frac{k}{1-\me W^n}\sum_{r=1}^{\lfloor n/k\rfloor -1}{n\choose rk}\int_0^1 x^{n-rk}(1-x)^{rk} f(x){\rm d}x\\
&\leq& {\rm const}\,\frac{k}{1-\me W^n}\sum_{r=1}^{\lfloor n/k\rfloor -1}{n\choose rk}\int_0^1 x^{n-rk-\alpha}(1-x)^{rk-\alpha}{\rm d}x\\
&=&{\rm const}\,\frac{k}{1-\me W^n}\sum_{r=1}^{\lfloor n/k\rfloor -1}\frac{\Gamma(n+1)\Gamma(n-rk-\alpha+1)\Gamma(rk-\alpha+1)}{\Gamma(n-2\alpha+2)
\Gamma(n-rk+1)\Gamma(rk+1)}\\
&\leq& {\rm const}\, \frac{1}{1-\me W^n}\frac{k}{n^{1-2\alpha}}\sum_{r=1}^{\lfloor n/k\rfloor -1}\left((n-rk)rk\right)^{-\alpha}\\
&\leq& {\rm const}\, \frac{1}{1-\me W^n}\frac{k^{1-2\alpha}}{n^{1-2\alpha}}\sum_{r=1}^{\lfloor n/k\rfloor -1}\left((\lfloor  n/k\rfloor -r)r\right)^{-\alpha}=O(1),\\
\end{eqnarray*}
The fourth line is a consequence of the inequality given in
\cite{Abr}, formula (6.1.47): for $c, d>-1$ there exists
$M_{c,d}>0$ such that for all $n\in\mn$
$$\bigg|{\Gamma(n+c)\over \Gamma(n+d)}-n^{c-d}\bigg|\leq M_{c,d}n^{c-d-1}.$$
The equality in the last line follows from the estimate
$\sum_{j=1}^{m-1}((m-j)j)^{-\alpha}\leq {\rm const}\;
m^{1-2\alpha}$ which holds for $\alpha<1$ and $m\in\mn$. The proof
of part (i) is complete.

Passing to part (ii) 
we can write
\begin{equation}\label{dens_repres}
f(x)=\mmp\{W \leq \delta_1\}f_1(x)+\mmp\{\delta_1< W\leq
1-\delta_2\}f_2(x)+\mmp\{W > 1-\delta_2\}f_3(x),
\end{equation}
where $f_1$, $f_2$ and $f_3$ are some densities such that $f_1$ is
nonincreasing on $(0,1)$, $f_3$ is nondecreasing on $(0,1)$ and
$f_2$ is bounded on $(0,1)$. It is known (see \cite{Lukacs}) that
if a random variable  $X$ with support $[0,1]$ has a nonincreasing
(nondecreasing) density $h$ then there exists a distribution
function $G$ such that $h(x)=\int_x^1 {{\rm d}G(y)\over y}$ (resp.
$h(x)=\int_{1-x}^1 {{\rm d}G(y)\over y}$).
Using this observation \eqref{dens_repres} can be rewritten as
follows
\begin{eqnarray*}
f(x)&=&\mmp\{W \leq \delta_1\} \int_0^1 \frac{1_{\{x\in[0,y]\}}}{y}{\rm d}G_1(y)+\mmp\{\delta_1< W\leq 1-\delta_2\}f_2(x)\\
&+&\mmp\{W > 1-\delta_2\} \int_0^1
\frac{1_{\{x\in[1-y,1]\}}}{y}{\rm d}G_2(y),
\end{eqnarray*}
where $G_1,G_2$ are some distribution functions concentrated on $[0,\delta_1]$ and $[1-\delta_2,1]$, respectively.

The last formula can be seen as a representation of $f$ as a
convex linear combination of the densities of three types:
$g_{\varepsilon}(x)=\varepsilon^{-1}{1_{\{x\in[0,\varepsilon]\}}}$,
$h_{\varepsilon}(x)=\varepsilon^{-1}{1_{\{x\in[1-\varepsilon,1]\}}}$
and bounded densities. Thus to prove (ii) it is enough to show
that relation \eqref{inhomogeneous_term1} holds for densities of
these types uniformly in $\varepsilon\in(0,1)$. The validity of
\eqref{inhomogeneous_term1} for bounded densities follows from
part (i) of the lemma (take $\alpha=0$). We only check
\eqref{inhomogeneous_term1} for $g_{\varepsilon}$, as the argument
is symmetric for $h_{\varepsilon}$. We have
$$
\mmp\{A_n=k\}={n\choose k}\varepsilon^{-1}\int_0^{\varepsilon}p^k(1-p)^{n-k}{\rm d}p=\frac{1}{(n+1)\varepsilon}I_{\varepsilon}(k+1,n-k+1),
$$
where $I_{\varepsilon}(k+1,n-k+1)$ is the normalized truncated
beta-function (see formula (6.6.2) in \cite{Abr}). Using formulae
(6.6.5) and (6.6.4) of the same reference we obtain
$$
\mmp\{A_n=k\}=\frac{1}{n+1}I_{\varepsilon}(k,n-k+1)+\frac{1-\varepsilon}{(n+1)\varepsilon}\mmp\{B\geq k+1\}\leq \frac{1}{n+1}+\frac{1}{(n+1)\varepsilon}\mmp\{B\geq k+1\}
$$
where a random variable $B$ has the binomial distribution with
parameters $(n,\varepsilon)$. This yields
\begin{eqnarray*}
k\sum_{r=1}^{\lfloor n/k\rfloor -1}\mmp\{A_n=rk\}&\leq& k\sum_{r=1}^{\lfloor n/k\rfloor -1}\Big(\frac{1}{n+1}+\frac{1}{(n+1)\varepsilon}\mmp\{B\geq rk+1\}\Big)\\
&\leq& 1 + \frac{k}{(n+1)\varepsilon}\sum_{r=1}^{\lfloor n/k\rfloor -1}\mmp\{B\geq rk+1\}\\
&=& 1 + \frac{k}{(n+1)\varepsilon}\sum_{r=1}^{\lfloor n/k\rfloor -1}\sum_{j=rk+1}^n\mmp\{B=j\}\\
&\leq& 1 + \frac{k}{(n+1)\varepsilon}\sum_{j=1}^n \sum_{r=1}^{\lfloor j/k\rfloor }\mmp\{B=j\}\\
&\leq& 1 + \frac{1}{(n+1)\varepsilon}\sum_{j=1}^n j\mmp\{B=j\}\leq 2.\\
\end{eqnarray*}
The proof of part (ii) is complete.
\end{proof}

\end{lemma}

\begin{lemma}\label{recursion_l}
Let $(b_n(k))_{n\in\mn,1\leq k\leq n}$, $(c_n)_{n\in\mn}$ and
$(d_n)_{n\in\mn}$ be nonnegative arrays.
Let $(a_n(k))_{n\in\mn_0,\,k\in\mn}$
and  $(a'_n)_{n\in\mn_0}$ be defined recursively via
\begin{eqnarray*}
a_0(k)=a_1(k)=\ldots=a_{k-1}(k)=0, \ \ k\in\mn;\\
a_n(k)=b_n(k)+\sum_{i=k}^{n-1}p_{n,i}a_i(k), \ \ k\leq n, \
k\in\mn;
\end{eqnarray*}
and
\begin{equation*}
a'_0=0,\;\;a'_n=d_n+\sum_{i=0}^{n-1}p_{n,i}a'_i,\;\;n\in\mn,
\end{equation*}
respectively, where $(p_{n,k})_{0\leq k\leq n-1}$ is a probability
distribution, for every fixed $n\in\mn$.

If
\begin{equation}\label{lcond}
c_k b_n(k)\leq d_n, \ \ n\in\mn,\;\;k\leq n, \ k\in\mn,
\end{equation}
then
\begin{equation}\label{lconc}
c_k a_n(k)\leq a'_n, \ \ n\in\mn,\;\;k\leq n, \ k\in\mn.
\end{equation}
\begin{proof}
We shall prove the lemma by induction on $n$. The base of
induction is straightforward. Assume that \eqref{lconc} holds for
all positive integer $n\leq N$ and $k\leq n$. We have to prove
\eqref{lconc} for $n=N+1$ and $k\leq N+1$, $k\in\mn$. Assume first
that $k\leq N$, then
\begin{eqnarray*}
c_k a_{N+1}(k)&=&c_k b_{N+1}(k)+\sum_{i=k}^{N}p_{N+1,\,i}c_k a_i(k)\overset{\eqref{lcond}}{\leq} d_{N+1}+
\sum_{i=k}^{N}p_{N+1,\,i}c_k a_i(k)\\
&\overset{\text{induction
}}{\leq} & d_{N+1}+\sum_{i=k}^{N}p_{N+1,\,i}a'_i\leq
d_{N+1}+\sum_{i=0}^N p_{N+1,\,i}a'_i=a'_{N+1}.
\end{eqnarray*}
For $k=N+1$ we have
$$
c_{N+1}a_{N+1}(N+1)=c_{N+1}b_{N+1}(N+1)\leq d_{N+1}\leq a'_{N+1}.
$$
The proof is complete.
\end{proof}
\end{lemma}

\vskip0.2cm \noindent {\bf Acknowledgement} The authors are
indebted to a referee for  thoughtful comments which have led to a
number of improvements. The work of the second author was
partially supported by the Department of Mathematics of the
Utrecht University and the Dutch stochastics cluster STAR.

\end{document}